\documentclass[11pt,reqno,a4paper]{article} 

\usepackage[english]{babel}
\usepackage[centertags]{amsmath}
\usepackage{amsfonts,amssymb,amsthm}
\usepackage{color}
\usepackage{hyperref} 				
\usepackage{a4wide}
\usepackage{float}
\usepackage{graphicx,cite}
\usepackage{palatino}
\usepackage{tikz,fp,ifthen,fullpage}
\usetikzlibrary{backgrounds}
\usetikzlibrary{decorations.pathmorphing,backgrounds,fit,calc,through}
\usetikzlibrary{arrows}
\usetikzlibrary{shapes,decorations,shadows}
\usetikzlibrary{fadings}
\usetikzlibrary{patterns}
\usetikzlibrary{mindmap}
\usetikzlibrary{decorations.text}
\usetikzlibrary{decorations.shapes}

\numberwithin{equation}{section}
\theoremstyle{plain}
\newtheorem{theorem}{Theorem}[section]
\newtheorem{proposition}[theorem]{Proposition}
\newtheorem{lemma}[theorem]{Lemma}

\newtheorem{coro}[theorem]{Corollary}
\newtheorem{corollary}[theorem]{Corollary}
\newtheorem{definition}[theorem]{Definition}

\theoremstyle{definition}
\newtheorem{remarkbase}[theorem]{Remark}
\newenvironment{remark}{\pushQED{\qed} \remarkbase}{\popQED\endremarkbase}

\newcommand{\R}{{\mathbb R}}

\newcommand{\mG}{\mathcal{G}}

\newcommand{\mL}{\mathcal{L}}

\newcommand{\mT}{\mathcal{T}}

\renewcommand{\a}{\alpha}
\renewcommand{\b}{\beta}
\newcommand{\g}{\gamma}
\renewcommand{\d}{\delta}

\newcommand{\e}{\varepsilon}
\newcommand{\ph}{\varphi}

\renewcommand{\t}{\tau}
\renewcommand{\th}{\vartheta}

\newcommand{\la}{\langle}
\newcommand{\ra}{\rangle}

\newcommand{\pa}{\partial}

\newcommand{\bcr}{\begin{color}{red}}
\newcommand{\ec}{\end{color}}
\newcommand{\bcb}{\begin{color}{blue}}
\newcommand{\bcg}{\begin{color}{green}}

\begin{document}

\title{Existence of a lens--shaped cluster of surfaces self--shrinking by mean curvature}

\author{Pietro Baldi, Emanuele Haus and Carlo Mantegazza}
\date{\today}

\maketitle

\begin{abstract}
We rigorously show the existence of a rotationally and centrally symmetric ``lens--shaped'' cluster of three surfaces, meeting at a smooth common circle, forming equal angles of $120$ degrees, self--shrinking under the motion by mean curvature.
\end{abstract}


\section{Introduction}
\label{sec:intro}

The goal of this paper is to give a contribution to the problem of the geometric evolution of a cluster of ($2$--dimensional) surfaces in $\R^3$, moving with normal velocity given by the mean curvature (at each interior point of every surface), meeting in triple along smooth common curves forming equal angles of $120$ degrees and meeting at common isolated points in groups of six with suitable angles (the ones that we get at the barycenter of a regular tetrahedron considering the six space triangles having as vertices such barycenter and any pair of vertices of the tetrahedron). Schulze and White called this flow ``mean curvature flow with triple edges'', see~\cite{schulzewhite}. Such geometric conditions are justified by ``energetic reasons'' and the physical model that we have in mind is a cluster of soap bubbles whose motion is driven only by the surface tension. However, this problem also modelizes the evolution of a $3$--dimensional system of different materials, where the energy is only given by the total area of the interfaces between them (see for instance~\cite{hermul,gurtin2,kinderliu} and {\em http://mimp.materials.cmu.edu}).

Even if the analogous evolution problem for systems of curves in the plane was considered by several authors~\cite{mannovtor,BeNo,schn-schu,schnurerlens,haettenschweiler,mazsae,Ilnevsch,MMN13,pluda,mannovplusch}, the literature on the surface case is actually quite small. In the papers~\cite{brakke,kimton} a global weak solution in the ``Brakke sense'' is constructed, while the short time existence of a smooth and regular solution has been established in~\cite{degako,schulzewhite} in some special cases. Anyway, the analysis of singularities and the subsequent possible restarting procedure are still open problems.
We also mention the works~\cite{freire2,freire3} where a graph evolving by mean curvature and meeting a horizontal plane with a fixed angle of $60$ degrees is studied: by considering the union 
of such a graph with its reflection through the plane, 
one gets an evolving symmetric ``lens--shaped'' domain, 
then the two boundary bounded surfaces, together with the unbounded part of the horizontal plane outside the domain given by the projection of the graph on such plane, 
give a cluster of three surfaces forming angles of $120$ degrees at the common curve given by the boundary of such domain. The clusters we consider in this paper belong to this class.

By the work of Huisken~\cite{huisk3} in the
smooth case of hypersurfaces in the Euclidean space and of
Ilmanen~\cite{ilman1,ilman3} in the more general weak setting of
varifolds, a suitable sequence of rescalings of the {\em subsets}
of $\R^n$ which are evolving by mean curvature (possibly generalized), approaching a singular time of the flow, ``converges'' to a so called ``blow--up limit'' set which, letting it flow again by mean curvature, simply moves by homothety, precisely, it shrinks down self--similarly toward the origin of the Euclidean space. This procedure and the classification of these special sets (possibly under some hypotheses), which are called {\em shrinkers}, is a key point in understanding the asymptotic behavior of the flow at a singular time.

Restricting ourselves to the case of the evolution of a {\em double--bubble} with a vertical axis of rotational symmetry, that is, a cluster of only three smooth surfaces meeting along a circle with angles of $120$ degrees (see~\cite[Section~9]{freire3}), which is the topologically simplest cluster in the compact case -- an example is shown in Figure 1, 
\begin{figure}[H]
\begin{center}
\begin{tikzpicture}[scale=.8]
\draw[color=black, thick,scale=1,domain=-1.57: 1.57,
smooth,variable=\t,shift={(0,0)},rotate=0]plot({1.*sin(\t r)},
{1.*cos(\t r)}) ; 
\filldraw[color=black!40!white, scale=1.2,domain=-1: 1,
smooth,variable=\t,shift={(0,0.55)},rotate=180]plot({1.*sin(\t r)},
{1.*cos(\t r)}) ; 
\filldraw[fill=white!83!black, 
scale=2.02,domain=-0.52: 0.52,
smooth,variable=\t,shift={(0,0.87)},rotate=180]plot({1.*sin(\t r)},
{1.*cos(\t r)});
\filldraw[color=white!83!black, scale=2.02,domain=-0.52: 0.52,
smooth,variable=\t,shift={(0,-0.87)},rotate=0]plot({1.*sin(\t r)},
{1.*cos(\t r)}) ; 
\draw[color=black!60!white,  scale=2.02,domain=-0.52: 0.52,
smooth,variable=\t,shift={(0,-0.87)},rotate=0]plot({1.*sin(\t r)},
{1.*cos(\t r)}) ; 
\draw[color=black!60!white,  scale=2.02,domain=-0.47: -0.18,
smooth,variable=\t,shift={(0,-0.91)},rotate=0]plot({1.*sin(\t r)},
{1.*cos(\t r)}) ; 
\draw[color=black!60!white,  scale=2.02,domain=0.18: 0.47,
smooth,variable=\t,shift={(0,-0.91)},rotate=0]plot({1.*sin(\t r)},
{1.*cos(\t r)}) ; 
\draw[color=black!60!white,  scale=2.02,domain=-0.425: -0.19,
smooth,variable=\t,shift={(0,-0.95)},rotate=0]plot({1.*sin(\t r)},
{1.*cos(\t r)}) ; 
\draw[color=black!60!white,  scale=2.02,domain=0.19: 0.425,
smooth,variable=\t,shift={(0,-0.95)},rotate=0]plot({1.*sin(\t r)},
{1.*cos(\t r)}) ; 
\draw[color=black!60!white,  scale=2.02,domain=-0.38: -0.20,
smooth,variable=\t,shift={(0,-0.99)},rotate=0]plot({1.*sin(\t r)},
{1.*cos(\t r)}) ; 
\draw[color=black!60!white,  scale=2.02,domain=0.20: 0.38,
smooth,variable=\t,shift={(0,-0.99)},rotate=0]plot({1.*sin(\t r)},
{1.*cos(\t r)}) ; 
\draw[color=black!60!white,  scale=2.02,domain=-0.31: -0.21,
smooth,variable=\t,shift={(0,-1.03)},rotate=0]plot({1.*sin(\t r)},
{1.*cos(\t r)}) ; 
\draw[color=black!60!white,  scale=2.02,domain=0.21: 0.31,
smooth,variable=\t,shift={(0,-1.03)},rotate=0]plot({1.*sin(\t r)},
{1.*cos(\t r)}) ; 
\draw[color= black!60!white, scale=2.02,domain=-0.52: 0.52,
smooth,variable=\t,shift={(0,0.87)},rotate=180]plot({1.*sin(\t r)},
{1.*cos(\t r)}) ; 
\draw[color= black!60!white, scale=2.02,domain=0.41: 0.51,
smooth,variable=\t,shift={(0,0.94)},rotate=180]plot({1.*sin(\t r)},
{1.*cos(\t r)}) ; 
\draw[color= black!60!white, scale=2.02,domain=0.38: 0.50,
smooth,variable=\t,shift={(0,1.01)},rotate=180]plot({1.*sin(\t r)},
{1.*cos(\t r)}) ; 
\draw[color= black!60!white, scale=2.02,domain=0.3: 0.48,
smooth,variable=\t,shift={(0,1.08)},rotate=180]plot({1.*sin(\t r)},
{1.*cos(\t r)}) ; 
\draw[color= black!60!white, scale=2.02,domain=0.32: 0.45,
smooth,variable=\t,shift={(0,1.15)},rotate=180]plot({1.*sin(\t r)},
{1.*cos(\t r)}) ; 
\draw[color= black!60!white, scale=2.02,domain=0.30: 0.41,
smooth,variable=\t,shift={(0,1.22)},rotate=180]plot({1.*sin(\t r)},
{1.*cos(\t r)}) ; 
\draw[color= black!60!white, scale=2.02,domain=0.23: 0.35,
smooth,variable=\t,shift={(0,1.29)},rotate=180]plot({1.*sin(\t r)},
{1.*cos(\t r)}) ; 
\draw[color= black!60!white, scale=2.02,domain=0.18: 0.30,
smooth,variable=\t,shift={(0,1.36)},rotate=180]plot({1.*sin(\t r)},
{1.*cos(\t r)}) ; 
\draw[color= black!60!white, scale=2.02,domain=0.12: 0.21,
smooth,variable=\t,shift={(0,1.43)},rotate=180]plot({1.*sin(\t r)},
{1.*cos(\t r)}) ; 
\draw[color= black, thick, scale=2.02,domain=-0.52: 0.52,
smooth,variable=\t,shift={(0,0.87)},rotate=180]plot({1.*sin(\t r)},
{1.*cos(\t r)}) ; 
\draw[color= black!60!white, scale=2.02,domain=1.35: 0.8,
smooth,variable=\t,shift={(0,-0.65)},rotate=180]plot({1.*sin(\t r)},
{1.*cos(\t r)}) ; 
\draw[color= black!60!white, scale=2.02,domain=1.31: 0.8,
smooth,variable=\t,shift={(0,-0.4)},rotate=180]plot({1.*sin(\t r)},
{1.*cos(\t r)}) ; 
\draw[color= black!60!white, scale=2.02,domain=1.2: 0.8,
smooth,variable=\t,shift={(0,-0.2)},rotate=180]plot({1.*sin(\t r)},
{1.*cos(\t r)}) ; 
\draw[color= black!60!white, scale=2.02,domain=1.1: 0.7,
smooth,variable=\t,shift={(0,0)},rotate=180]plot({1.*sin(\t r)},
{1.*cos(\t r)}) ; 
\draw[color= black!60!white, scale=2.02,domain=1: 0.7,
smooth,variable=\t,shift={(0,0.2)},rotate=180]plot({1.*sin(\t r)},
{1.*cos(\t r)}) ; 
\draw[color=black!60!white, scale=2.02,domain=0.91: 0.65,
smooth,variable=\t,shift={(0,0.35)},rotate=180]plot({1.*sin(\t r)},
{1.*cos(\t r)}) ; 
\draw[color= black!60!white, scale=2.02,domain=0.82: 0.43,
smooth,variable=\t,shift={(0,0.5)},rotate=180]plot({1.*sin(\t r)},
{1.*cos(\t r)}) ; 
\draw[color= black!60!white, scale=2.02,domain=0.73: 0.4,
smooth,variable=\t,shift={(0,0.63)},rotate=180]plot({1.*sin(\t r)},
{1.*cos(\t r)}) ; 
\draw[color=black!60!white, scale=2.02,domain=0.66: 0.43,
smooth,variable=\t,shift={(0,0.71)},rotate=180]plot({1.*sin(\t r)},
{1.*cos(\t r)}) ; 
\draw[color=black,thick, scale=1.98,domain=-2.61: 2.61,
smooth,variable=\t,shift={(0,-0.86)},rotate=180]plot({1.*sin(\t r)},
{1.*cos(\t r)}) ; 
\draw[color=black, thick,scale=1,domain=-1.57: 1.57,
smooth,variable=\t,shift={(0,0)},rotate=0]plot({1.*sin(\t r)},
{1.*cos(\t r)}) ; 
\draw[dashed]
(0,1.7)--(0,-4.5);
\end{tikzpicture}
\end{center}
\begin{caption}{A rotationally symmetric {\em double--bubble}.\label{fig1}}
\end{caption}
\end{figure}
\noindent one hypothetical possibility, after performing the above blow--up procedure (if only one bubble is ``collapsing'') at a singular time of the flow, is to obtain a rotationally symmetric ``lens shaped'' shrinker, which is also symmetric with respect to the horizontal plane, formed by two bounded surfaces (one the horizontal reflection of the other) and an unbounded plane surface, meeting at angles of $120$ degrees along a common circle in the plane. Such a shrinker is then given by the rotation along the vertical axis of a``lens shaped'' network of three curves, symmetric with the respect to the horizontal and the vertical line and meeting the horizontal line at an angle of $60$ degrees, as in Figure \ref{fig2}. 
\begin{figure}[H]
\begin{center}
\begin{tikzpicture}[scale=0.2]
\filldraw[pattern=north east lines,  pattern color=black!30!white, scale=6, shift={(1.53,4.5)}]
(-5,-6)--(0,-6)--
(0,-6)--(2.5,-3.5)--
(2.5,-3.5)--(-2.5,-3.5)--
(-2.5,-3.5)--(-5,-6);
\draw[color=white, very thick, scale=6, shift={(1.53,4.5)}]
(-5,-6)--(0,-6)
(0,-6)--(2.5,-3.5)
(2.5,-3.5)--(-2.5,-3.5)
(-2.5,-3.5)--(-5,-6);
\filldraw[ball color=white]
(-3.035,0)to[out= 60,in=180, looseness=1] (2.2,2.8)--
(2.2,2.8)to[out= 0,in=120, looseness=1] (7.435,0)--
(7.435,0) to[out= -135,in=0, looseness=1](0.36,-1.83)--
(0.36,-1.83) to[out= 180,in=-135, looseness=1] (-3.035,0);
\draw[color=black]
(-3.035,0)to[out= -135,in=180, looseness=1](0.36,-1.83)
(0.36,-1.83)to[out= 0,in=-135, looseness=1] (7.435,0)
(-3.035,0)to[out= 60,in=180, looseness=1] (2.2,2.8)
(2.2,2.8)to[out= 0,in=120, looseness=1] (7.435,0);
\draw[color=black, dashed]
(2.2,-2.7)to[out= 0,in=-120, looseness=1] (7.435,0)
(-3.035,0)to[out= -60,in=180, looseness=1] (2.2,-2.7);
\draw[color=black]
(-7,0)to[out= 0,in=180, looseness=1](-3.035,0)
(7.435,0)to[out= 0,in=180, looseness=1](11.4,0);
\draw[color=black,dashed]
(-9,0)to[out= 0,in=180, looseness=1](-7,0)
(11.4,0)to[out= 0,in=180, looseness=1](13.4,0);
\draw[black!70!white, dashed, scale=6, shift={(3.87,4.5)}]
(-5,-6)--(-2.5,-3.5);
\draw[black!70!white, dashed]
(-13,0)--(18,0)
(2.2,8)--(2.2,-8);
\path[font=\normalsize]
(-7,-9)node[above]{$x$}
(16,-0.2)node[above]{$y$}
(3,6.25)node[above]{$z$};
\end{tikzpicture}\begin{tikzpicture}[scale=0.22]
\draw[color=black, thick]
(-3.035,0)to[out= 60,in=180, looseness=1] (2.2,2.8)
(2.2,2.8)to[out= 0,in=120, looseness=1] (7.435,0);
\draw[color=black, thick]
(2.2,-2.7)to[out= 0,in=-120, looseness=1] (7.435,0)
(-3.035,0)to[out= -60,in=180, looseness=1] (2.2,-2.7);
\draw[color=black, thick]
(-11,0)to[out= 0,in=180, looseness=1](-3.035,0)
(7.435,0)to[out= 0,in=180, looseness=1](15.4,0);
\draw[color=black,dashed]
(-13,0)to[out= 0,in=180, looseness=1](-11,0)
(15.4,0)to[out= 0,in=180, looseness=1](17.4,0);
\draw[black!70!white, dashed]
(-13,0)--(18,0)
(2.2,8)--(2.2,-8);
\path[font=\normalsize]
(16,-0.2)node[above]{$y$}
(3,7)node[above]{$z$};
\draw[-latex] ($(0, 0) + (20:1.4cm and 20cm)$(P) arc(110:430:2.5cm and .7cm);
\end{tikzpicture}
\end{center}
\begin{caption}{A rotationally symmetric ``lens shaped'' cluster.\label{fig2}}
\end{caption}
\end{figure}
The goal of this paper is to show rigorously the existence of a shrinker with such a structure. In the case of networks, the work~\cite{schnurerlens} includes global existence results for symmetric systems of curves of ``lens type'' that are shrinkers. Other classification results (the classification is complete only in low--complexity classes of networks) can be found in~\cite{balhausman,balhausman3,chenguo,schnurerlens,haettenschweiler}. 
In the $2$--dimensional ``cluster of surfaces'' case, to our knowledge, the present paper is the first in the literature about existence or non--existence of shrinkers with a given topological/geometric structure.

We underline that the analogous ``standard'' classification problem for shrinkers given by a single smooth, embeded surface (compact or not), rotationally symmetric with respect to the vertical axis, 
is quite open. The only known examples are the plane, the sphere and the {\em Angenent's torus} in~\cite{angen4}. Moreover, we mention that our work is particularly related to the analysis in~\cite[Section~9]{freire3} of the 
evolution by mean curvature of rotationally symmetric clusters of three surfaces, since any shrinker arising from the blow--up procedure at a singular time in such a situation must have 
the same structure and symmetries as the ones we consider in this paper.

Our main result is the following.

\begin{theorem} 
\label{thm:main} There exists a rotationally symmetric ``lens shaped'' shrinker, which is also symmetric with respect to the horizontal plane containing its unbounded surface, such that the three surfaces meet at a common circle in the plane forming three angles of $120$ degrees (as in Figure~\ref{fig2}).
\end{theorem}

Let us introduce the analytic setting we use to prove this theorem. 
Let $\gamma:I\to\R^2$ be the curve depicted in Figure~\ref{fig2}. 
As it is well known, at every point $x\in\R^3$ of a shrinker $S$ the equation 
\begin{equation*}
\overline{\mathrm{H}}+x^\perp=0
\end{equation*}
must be satisfied, 
where $\overline{\mathrm{H}}$ is the mean curvature vector of $S$ and $x^\perp$ denotes the projection of the position vector $x$ on the normal space to $S$. 
By the assumed rotational symmetry, it is easy to compute the equation that the curve $\gamma$ must satisfy (which is actually a system of ODEs for $\gamma$), obtaining
(see~\cite[Section~9]{freire3} where such an equation is computed writing $\gamma$ as a graph on the horizontal line)
\begin{equation} \label{geo.1}
k + \g \cdot \nu + \frac{ \g_s \cdot e_2 }{ \g \cdot e_1 } = 0,
\end{equation}
where $\{e_1,e_2\}$ is the canonical basis of $\R^2$, 
$e_2$ is the unit vector of the rotation axis, 
and, introducing an arclength parameter $s$ on the curve $\gamma$, 
the unit tangent vector is $\tau = \gamma_s = \frac{d\,}{ds}\gamma$, 
the unit normal vector $\nu$ is the counterclockwise rotation of $\pi/2$ in $\R^2$ of the vector
$\tau$, and finally, if $\overline{k}=\gamma_{ss}=\frac{d^2\,}{ds^2}\gamma$ is the curvature vector of $\gamma$, we set $k=\overline{k}\cdot\nu$.

Moreover, in order that the rotation of the curve $\gamma$ around the vertical axis describes a smooth surface intersecting the horizontal plane with an angle of $60$ degrees, 
we need to require that 
\begin{itemize}
\item $\gamma(0)\cdot e_1 =0$,
\item $\gamma_s(0)\cdot e_2=0$ (horizontal tangent at the intersection with the vertical axis),
\item $\gamma(L)\cdot e_2 =0$,
\item $\gamma_s(L)\cdot e_1=1/2$ 
(exterior angle of $120$ degrees with the horizontal plane),
\item $\g(s) \cdot e_2 > 0$ for all $s \in [0,L)$,
\item $\g(s) \cdot e_1 > 0$ for all $s \in (0,L]$,
\item $\g$ has no self--intersections,
\end{itemize}
where $L$ is the (a priori unknown) length of the curve $\gamma$.

\begin{remark}\label{2222}
In the analogous one--dimensional problem of a network of curves in the plane moving by curvature (see~\cite{mannovtor,mannovplusch}, for instance), every curve $\gamma:I\to\R^2$ of a shrinking network must satisfy the equation 
\begin{equation}\label{1111}
k+\gamma\cdot\nu=0
\end{equation}
(see the works by Abresch--Langer~\cite{ablang1} and of
Epstein--Weinstein~\cite{epswei}, where a full classification in the case of a single closed curve is obtained). The ``extra term'' $\frac{ \g_s \cdot e_2 }{ \g \cdot e_1 } $ in equation~\eqref{geo.1} makes all the analysis more complicated, since its presence ``destroys'' the existence of a constant quantity (a first integral of the system of ODEs) along the curve which can be derived for the solutions of equation~\eqref{1111} (see~\cite[Section~2]{balhausman}).
\end{remark}

We set some notation that will be used through all the paper. We denote with $u$ and $v$ the components of the curve $\gamma$, that is, $\g(s) = (u(s), v(s))$, where $s$ is the arclength parameter
\begin{equation} \label{geo.2}
u'^2(s) + v'^2(s) = 1,
\end{equation}
and $'$ denotes $\frac{d}{ds}$. 
The curvature is given by 
\begin{equation} \label{geo.3}
k = \g_{ss} \cdot \nu = \begin{pmatrix} u'' \\ v'' \end{pmatrix}
\cdot \begin{pmatrix} -v' \\ u' \end{pmatrix}
= - v' u'' + u' v'',
\end{equation}
then equation~\eqref{geo.1} becomes 
\begin{equation} \label{geo.4}
- v' u'' + u' v'' + \frac{ v' }{ u } - u v' + v u' = 0
\end{equation}
(the term $v' / u$ in equation~\eqref{geo.4} is the ``extra term'' with respect to the corresponding equation for shrinking networks, which we mentioned in Remark~\ref{2222}).
Equation~\eqref{geo.4}, together with~\eqref{geo.2}, can be written 
as a system of two second order ODEs in $(u,v)$: 
since $u' u'' + v' v'' = 0$, one has 
\begin{equation} \label{geo.5}
\begin{cases} 
u'' = - v'^2 \big( u - \frac{1}{u} \big) + u' v' v 
\\
v'' = v' u' \big( u - \frac{1}{u} \big) - u'^2 v.
\end{cases}
\end{equation}

Theorem~\ref{thm:main} on the existence of a ``lens shaped'' shrinker 
is then a consequence of the following result.

\begin{proposition}
\label{question:main}
There exist $a > 0$, $\bar s > 0$, and $u, v : [0, \bar s] \to \R$ of class $C^2([0, \bar s])$ 
that solve~\eqref{geo.5} on $(0, \bar s]$ such that the curve $(u,v)$: 
\begin{itemize} 
\item 
intersects the vertical axis at the point $(0,a)$ with horizontal tangent, 
namely 
\begin{equation} \label{init cond}
u(0) = 0, \quad 
v(0) = a, \quad 
u'(0) = 1, \quad 
v'(0) = 0;
\end{equation}

\item
it reaches the horizontal axis forming an angle of $60$ degrees, namely 
\begin{equation} \label{touch cond}
u(\bar s) > 0, \quad 
v(\bar s) = 0, \quad 
u'(\bar s) = \frac12, \quad 
v'(\bar s) = - \frac{\sqrt3}{2};
\end{equation}

\item
it remains in the region $u > 0$, $v > 0$ on the interval $(0, \bar s)$ 
and it does not self--intersect.
\end{itemize}
\end{proposition}

\emph{Strategy of the proof.}
We prove Proposition~\ref{question:main} by a continuity argument. 
On one hand, it is known (and immediate to check) that the circle 
$(u(s), v(s)) = ( \sqrt2 \, \sin(s/\sqrt2) ,$ $\sqrt2 \, \cos(s/\sqrt2) )$, 
$s \in [0, \pi / \sqrt2]$ solves~\eqref{geo.5}--\eqref{init cond} with $a = \sqrt2$, 
and it intersects the horizontal axis with an angle of 90 degrees. 
On the other hand, for $a$ small the solution of the nonlinear equation~\eqref{geo.5} 
is close to the solution of the corresponding linear problem, therefore, by a perturbation analysis, we prove that the curve intersects 
the horizontal axis with a small angle. 
As a consequence, we deduce Proposition~\ref{question:main} for some $a \in (0, \sqrt2)$.

To realize this strategy, we have to deal with several points.
First, we prove local existence for the Cauchy problem~\eqref{geo.5}--\eqref{init cond}. 
This does not follow directly from the classical ODEs theory, because the problem is \emph{degenerate} as $u \to 0$. 
In Section~\ref{sec:LWP} we write $(u,v)$ as the graph of a function $y = f(x)$ 
(without loss of generality, at least locally, by~\eqref{init cond})
and we prove existence and uniqueness for the problem~\eqref{1010.1}-\eqref{init cond f} 
for $f(x)$ in class $C^2$, 
making $\pa_{xx} + \frac{1}{x} \pa_x$ play the r\^ole of the highest order operator, 
see~\eqref{2305.1}-\eqref{def mT}.

Second, we analyse the case of $a$ small. 
It is natural to expect that the smaller is $a$, the longer is the existence interval of the solution. However, the interval given by the 
$C^2$ existence result of Section~\ref{sec:LWP} does not go beyond a fixed threshold. 
This is not sufficient for our perturbation analysis, 
because such a threshold is smaller than the point at which the solution of the linear problem 
reaches the horizontal axis. 
We get around this issue by observing that the entire linear part~\eqref{def mL} 
of the equation can be inverted in a suitable space of analytic functions, 
see Lemma~\ref{lemma:mL inv}. 
As a consequence, we get a longer existence interval for solutions with small $a$. 
Hence in the analytic class we are able to carry out the perturbation analysis, 
proving that for $a$ small the solution intersects the horizontal axis with a small angle. 
This is the content of Section~\ref{sec:analytic}.

Third, we have to prove that for all $a \in (0, \sqrt2)$ 
the continuation of the local solution of~\eqref{geo.5}-\eqref{init cond} 
constructed in Section~\ref{sec:LWP} 
reaches the horizontal axis without self-intersecting and without touching the vertical axis.
In Section~\ref{sec:extension graph} we prove that for all $a > 0$ 
the solution of the Cauchy problem remains the graph of a function $y = f(x)$ 
in the vertical strip $x \in [0,1]$, using classical comparison arguments for ODEs. 
Then, the analysis of Section~\ref{sec:TG} is in some sense the core and the most original part of the proof. To begin with, we obtain a useful integral formula for the curvature $k$, 
see~\eqref{u.17}. An interesting feature of this formula is that it incorporates the initial condition~\eqref{init cond}, in the sense that~\eqref{u.17} is not satisfied by all the solutions 
of equation~\eqref{geo.5}, but only by those starting from the vertical axis with horizontal 
tangent vector. 
Using formula~\eqref{u.17}, we prove the lower bound~\eqref{u.18.glob}, 
which is a quantitative version of the transversality 
of the tangent vector $\g_s$ with respect to the position vector $\g$. 
As a consequence, we deduce that the solution reaches the horizontal axis 
without self-intersecting and without touching the vertical axis. 
The continuous dependence of the angle at the horizontal axis on the initial height $a$ 
is also a consequence of the transversality property above. 

\medskip

As a final remark, we underline that our proof, which is merely based on the continuity 
of the function that maps the height $a$ to the intersection angle, 
does not give any information about the uniqueness 
of such rotationally symmetric ``lens shaped'' shrinker, 
which remains an open question. Moreover, it seems quite difficult to classify all the shrinkers with this ``lens'' structure (topologically), even assuming for instance convexity (or non--negativity of the mean curvature) of the three surfaces, we actually conjecture that they all should be rotationally symmetric.

\bigskip

\emph{Acknowledgements.} 
We thank Matteo Novaga and Alessandra Pluda for several comments and discussions, 
and Giacomo Ascione for some useful numerical simulations.
We also thank Alessandra Pluda for drawing the figures.

\section{General local existence}
\label{sec:LWP}

As a starting point of our analysis, we look for a profile $(u,v)$ 
that is the graph of a function, 
\begin{equation} \label{geo.6}
v(s) = f(u(s)).
\end{equation}
From~\eqref{geo.6} one has
\begin{equation} \label{geo.7}
v' = f'(u) u', \quad 
v'' = f''(u) u'^2 + f'(u) u'',
\end{equation}
and, since $v'^2 = 1 - u'^2$, 
\begin{equation} \label{geo.8}
u'^2 (1 + f'^2(u)) = 1.
\end{equation}
Substituting in~\eqref{geo.4} 
(and renaming $u \to x$) 
we find the equation for the profile $f$, which is 
\begin{equation} \label{1010.1}
f''(x) = \big( 1 + f'^2(x) \big) \Big[ f'(x) \Big( x - \frac{1}{x} \Big) - f(x) \Big].
\end{equation}
Because of the coefficient $1/x$, equation~\eqref{1010.1} makes sense on $x > 0$. 
Thus we look for a function $f$ of class $C^2$ on some interval $[0,r]$ 
that satisfies the initial conditions 
\begin{equation} \label{init cond f}
f(0) = a, \quad f'(0) = 0
\end{equation}
and solves~\eqref{1010.1} on $x \in (0,r]$.

It is convenient to write $f(x)$ as $a + h(x)$. 
Therefore we look for classical solutions $h \in C^2([0,r])$ 
on some interval $[0,r]$ of the Cauchy problem
\begin{equation} \label{Cp h}
\begin{cases} 
h''(x) = \big( 1 + h'^2(x) \big) \big[ h'(x) \big( x - \frac{1}{x} \big) - h(x) - a \big] \\
h(0) = h'(0) = 0.
\end{cases}
\end{equation}

In this section we prove the existence and uniqueness in class $C^2$ 
of a local solution of the degenerate Cauchy problem~\eqref{Cp h} for any $a > 0$. 
We write~\eqref{Cp h} as 
\begin{equation} \label{2305.1}
\mT h = P(h,a)
\end{equation}
where 
\begin{align} 
\label{def mT}
\mT h (x) & := h''(x) + \frac{h'(x)}{x}, 
\\
P(h,a)(x) & := 
x h'(x) - h(x) - a + h'^2(x) \big[ h'(x) \big( x - \tfrac{1}{x} \big) - h(x) - a \big]. 
\label{def P}
\end{align}
For $r>0$, we define
\[
C^2_0([0,r]) := \{ h \in C^2([0,r], \R) : h(0) = h'(0) = 0 \}
\]
with norm $\| h'' \|_{C^0([0,r])}$. 
We look for solutions of the nonlinear equation~\eqref{2305.1} in $C^2_0([0,r])$. 
We begin with studying the linear problem $\mT h = g$.

\begin{lemma} \label{lemma:mL0 inv} 
Let $r > 0$. 
For every $g \in C^0([0,r])$ there exists a unique $h \in C^2_0([0,r])$ 
such that $\mT h = g$. 
This defines the inverse operator 
$\mT^{-1} : C^0([0,r]) \to C^2_0([0,r])$, $g \mapsto h = \mT^{-1} g$. 
Moreover 
\begin{equation} \label{stima mL0 inv}
\| (\mT^{-1} g)'' \|_{C^0([0,r])} = \| h'' \|_{C^0([0,r])}
\leq \frac32 \| g \|_{C^0([0,r])}.
\end{equation}
\end{lemma}

\begin{proof}
Consider the equation $h'' + \frac{1}{x} h' = g$. 
By variation of constants,
\[
h'(x) = \frac{c}{x} + \frac{1}{x} \int_0^x t g(t) \, dt, \quad c \in \R.
\]
Since $g$ is continuous, the second term vanishes as $x \to 0$. 
Then the condition $h'(0) = 0$ determines $c=0$. 
Integrating by parts and using $h(0) = 0$, we find the solution
\begin{equation} \label{formula h}
h(x) = \int_0^x (\log x - \log t) t g(t) \, dt.
\end{equation}
Then 
$
h''(x) = - \frac{1}{x^2} \int_0^x t g(t) \, dt + g(x),
$
and~\eqref{stima mL0 inv} follows.
\end{proof}

\begin{lemma}
\label{thm:LWP generico C2}
Let $a, r, R, L$ be real positive numbers satisfying
\begin{align} \label{1311.4 C2}
& \frac32 a + \frac94 r^2 R + \frac32 a r^2 R^2 
+ \frac32 \Big( r^2 + \frac32 r^4 \Big) R^3
\leq R,
\\
& r^2 \Big( \frac94 + \frac92 R^2 + 3 a R^2 + \frac{27}{4} R^2 r^2 \Big) 
\leq L < 1.
\label{1311.5 C2}
\end{align}
Then there exists a unique function $h \in C^2_0([0,r])$ 
that solves the Cauchy problem~\eqref{2305.1}. 
Moreover $\| h'' \|_{C^0([0,r])} \leq R$.
As a consequence, the function $f(x) = a + h(x)$ solves equation~\eqref{1010.1} 
with initial data $f(0) = a$, $f'(0) = 0$. 
\end{lemma}

\begin{proof}
By Lemma~\ref{lemma:mL0 inv}, equation~\eqref{2305.1} 
can be written as the fixed point problem in $C^2_0([0,r])$
\begin{equation} \label{nonlin mL C2}
h = \Phi(h) \quad \ \text{where} \ 
\Phi(h) := \Phi(h,a) := \mT^{-1} P(h,a).
\end{equation}
We prove that, for suitable $r,R$, 
the map $\Phi$ is a contraction in the ball 
\begin{equation} \label{ball C2}
B_{R,r} := \{ h \in C^2_0([0,r]) : \| h'' \|_{C^0([0,r])} \leq R \}.
\end{equation}
By~\eqref{stima mL0 inv}, 
using the inequalities $|h'(x)| \leq x \| h'' \|_{C^0([0,r])}$ and 
$|h(x)| \leq \frac12 r^2 \| h'' \|_{C^0([0,r])}$, 
one has 
\[ 
\| (\Phi(h))'' \|_{C^0([0,r])} 
\leq \frac32 \| P(h,a) \|_{C^0([0,r])} 
\leq \frac32 a + \frac94 r^2 R + \frac32 a r^2 R^2 
+ \frac32 \Big( r^2 + \frac32 r^4 \Big) R^3,
\]
and
\[
\| ( \Phi(h) - \Phi(g) )'' \|_{C^0([0,r])} 
\leq r^2 \Big( \frac94 + \frac92 R^2 + 3 a R^2 + \frac{27}{4} R^2 r^2 \Big) 
\| (h-g)'' \|_{C^0([0,r])}.
\]
for all $h,g \in B_{R,r}$. 
Assumptions~\eqref{1311.4 C2}-\eqref{1311.5 C2} 
imply that $\Phi$ is a contraction of the ball $B_{R,r}$ into itself.
\end{proof}

\begin{lemma} \label{lemma:LWP}
Let $a > 0$. Let $r = r_a$, $R = R_a$ be defined by
\[
R_a := 6a, \quad 
r_a := \min \Big\{ \frac{1}{3\sqrt2}, \, 
\frac{1}{36 a}, \, 
\frac{1}{12 \sqrt 6 \, a^{3/2}} \Big\}.
\]
Then~\eqref{1311.4 C2}-\eqref{1311.5 C2} hold with $L = 1/2$. 

As a consequence, for every $a > 0$ there exists a unique solution $h \in C^2_0([0, r_a])$ 
of the Cauchy problem~\eqref{2305.1}, and it satisfies $\| h'' \|_{C^0([0, r_a])} 
\leq 6a$.
\end{lemma}

\begin{proof}
It is a straightforward check.
\end{proof}

\begin{remark} \label{rem:even}
If $h \in C^2_0([0,r])$ solves~\eqref{2305.1}, 
then $g(x) := h(-x)$ solves~\eqref{2305.1} on the interval $[-r,0]$. 
Hence the even extension of $h$ is the unique solution of~\eqref{2305.1} on $[-r,r]$. 
In other words, the solutions of~\eqref{2305.1} with initial data $h(0) = h'(0) = 0$ 
are all even functions.
\end{remark}

\begin{lemma} \label{lipschitz a}
$(i)$ Let $r_* := 1 / (36 \sqrt2)$, $R_* := 6 \sqrt2$. 
For all $a \in [0, \sqrt2]$, the solution $h_a$ of~\eqref{2305.1} in Lemma~\ref{thm:LWP generico C2}
belongs to $C^2_0([0,r_*])$, with $\| h_a'' \|_{C^0([0,r_*])} \leq R_*$.

$(ii)$ The map $[0, \sqrt2] \to C^2_0([0,r_*])$, $a \mapsto h_a$ is Lipschitz, 
more precisely
\[
\| (h_{a_1} - h_{a_2})'' \|_{C^0([0,r_*])} \leq \frac{37}{12} |a_1 - a_2| \quad \forall a_1, a_2 \in [0, \sqrt2].
\]
\end{lemma}

\begin{proof}
$(i)$ We apply Lemma~\ref{thm:LWP generico C2} and Lemma~\ref{lemma:LWP}. 
Since $a \mapsto r_a$ is non--increasing, 
we have $r_a \geq r_{\sqrt{2}} = r_*$ for all $a \in [0, \sqrt2]$ and
\[
\| h_a'' \|_{C^0([0,r_*])} 
\leq \| h_a'' \|_{C^0([0,r_a])}
\leq R_a = 6 a \leq 6 \sqrt 2 = R_*\,.
\]
$(ii)$ By Lemma~\ref{lemma:LWP}, for $a = \sqrt2$, $r = r_*$ and $R = R_*$, 
conditions~\eqref{1311.4 C2}-\eqref{1311.5 C2} hold with $L = 1/2$. 
In addition,~\eqref{1311.4 C2}-\eqref{1311.5 C2} also hold 
with $L = 1/2$, $r=r_*$, $R = R_*$, for all $a \in [0, \sqrt2]$.

Since $h_a$ is the fixed point of the map $\Phi_a := \Phi(\cdot, a) = \mT^{-1} P(\cdot, a)$, 
one has 
\begin{align*}
& \| (h_{a_1} - h_{a_2})'' \|_{C^0([0,r_*])} 
= \| \{ \Phi_{a_1} (h_{a_1}) - \Phi_{a_2} (h_{a_2}) \}'' \|_{C^0([0,r_*])}
\\ 
& \qquad \qquad \leq \| \{ \Phi_{a_1} (h_{a_1}) - \Phi_{a_1} (h_{a_2}) \}'' \|_{C^0([0,r_*])} 
+ \| \{ \Phi_{a_1} (h_{a_2}) - \Phi_{a_2} (h_{a_2}) \}'' \|_{C^0([0,r_*])} 
\\ 
& \qquad \qquad \leq \frac12 \| (h_{a_1} - h_{a_2})'' \|_{C^0([0,r_*])}
+ \| \{ \Phi_{a_1} (h_{a_2}) - \Phi_{a_2} (h_{a_2}) \}'' \|_{C^0([0,r_*])}.
\end{align*}
The term $\frac12 \| (h_{a_1} - h_{a_2})'' \|_{C^0([0,r_*])}$ 
is absorbed, so it remains to estimate the last term. 
For any $g \in C^2_0([0,r_*])$ with $\| g'' \|_{C^0([0,r_*])} \leq R_*$, 
we use Lemma~\ref{lemma:mL0 inv} to estimate
\begin{align*}
\| \{ \Phi_{a_1} (g) - \Phi_{a_2} (g) \}'' \|_{C^0([0,r_*])}
& = \| \{ \mT^{-1} [P(g, a_1) - P(g, a_2)] \}'' \|_{C^0([0,r_*])}
\\ & 
\leq \frac32 \| P(g, a_1) - P(g, a_2) \|_{C^0([0,r_*])}
\\ & 
= \frac32 \| 1+g'^2 \|_{C^0([0,r_*])} \, |a_1 - a_2|
\\ 
& \leq \frac32 (1 + r_*^2 R_*^2) |a_1 - a_2| = \frac{37}{24} |a_1 - a_2|\,.
\qedhere
\end{align*}
\end{proof}

\section{Analysis of small solutions}
\label{sec:analytic}

For small $a$, it is natural to expect the solutions to exist for longer times than 
$r_a = 1 / 3 \sqrt2$ given by Lemma~\ref{lemma:LWP}. 
Also, one expects the solutions of the nonlinear problem to stay close to the solutions
of the corresponding linear problem for a long time. 
To get this kind of estimates for $a$ small, 
and to obtain a precise description of the linear solutions, 
it is more convenient to work with power series. 
The point is that in power series we are able to invert the whole linear part of the equation, 
and not only $\mT$.

We write the nonlinear problem~\eqref{Cp h} as
\begin{equation} \label{nonlin mL.3}
\mL h = Q(h,a), \quad h(0) = h'(0) = 0
\end{equation}
where
\begin{align} \label{def mL}
\mL h(x) 
& := h''(x) + h'(x) \Big( \frac{1}{x} - x \Big) + h(x),
\\
Q(h,a) 
& := - a + (x - \tfrac{1}{x}) h'^3 - h'^2 (h+a).
\label{def Q}
\end{align}
For $r > 0$, we consider the space of even (recall Remark~\ref{rem:even}), 
real analytic functions 
\begin{equation} \label{def Yr}
X_r = \Big\{ f(x) = \sum_{n=0}^\infty f_n x^n : 
f_n \in \R \ \ \forall n, \ \ 
f_n = 0 \ \ \forall \text{$n$ odd}, 
\ \ \| f \|_r < \infty \Big\}
\end{equation}
with norm
\begin{equation} \label{def norma r}
\| f \|_r := \sum_{n=0}^\infty |f_n| \la n \ra r^{n-1}, \quad 
\la n \ra := \max \{ 1, n \},
\end{equation}
and its subspace
\begin{equation} \label{Xr0}
X_r^0 := \{ h \in X_r : h(0) = 0 \}.
\end{equation}

Given $f$ as in~\eqref{def Yr}, one has 
\begin{equation} \label{mL f}
\mL f(x) = \sum_{n=0}^\infty \big[ (n+2)^2 f_{n+2} - (n-1) f_n \big] x^n.
\end{equation}
Hence $\mL : X_r \to X_{r'}$ for all $r' \in (0,r)$. 

\begin{lemma} \label{lemma:kernel}
For all $0 < r_1 < r_2$, the operator $\mL : X_{r_2} \to X_{r_1}$ 
has a one--dimensional kernel $\{ c \eta : c \in \R\}$, where $\eta$ is the function 
\begin{equation} \label{def eta}
\eta(x) = \sum_{n=0}^\infty \eta_n x^n, 
\quad \eta_0 = 1, 
\quad \eta_{n+2} = \frac{n-1}{(n+2)^2} \,\eta_n \quad \forall n \geq 0,
\end{equation}
$\eta_n = 0$ for all odd $n$. 
Note that $\eta \in X_r$ for all $r > 0$. 
Also, $\eta_n < 0$ for all even $n \geq 2$. 
The difference 
\begin{equation} \label{def J}
J(x) := 1 - \eta(x) 
= - \sum_{n = 2}^\infty \eta_n x^n
= \sum_{n = 2}^\infty |\eta_n| x^n
\end{equation}
solves $\mL J(x) = 1$ for all $x > 0$, 
it satisfies 
$J(0) = J'(0) = 0$, 
$J''(0) = \frac12$, 
its derivatives $J^{(k)}(x)$ are $> 0$ for all $x > 0$, 
all $k=0,1,2,\ldots$, and 
$\| J \|_r \leq \frac12 r e^{r^2/2}$.
\end{lemma}

\begin{proof}
Let $\mL f = 0$. Then $f_n$ is recursively determined by $f_{n+2} = (n-1) (n+2)^{-2} f_n$ 
for all $n \geq 0$. Hence $f(x) = f_0 \eta(x)$, 
where $\eta$ is defined in~\eqref{def eta}. 

For $r>0$, let $\g_n := |\eta_n| n r^{n-1}$. 
Then, by~\eqref{def eta}, 
\[
\g_2 = \frac{r}{2}, \qquad 
\frac{\g_{n+2}}{\g_n} = \frac{(n-1) r^2}{(n+2)n} 
< \frac{r^2}{n+2} \quad \ \forall n = 2,4,6,\ldots,
\]
and
\begin{align*}
\| J \|_r & = \sum_{n=2}^\infty |\eta_n| n r^{n-1} 
= \g_2 + \g_4 + \g_6 + \ldots 
= \g_2 \Big( 1 + \frac{\g_4}{\g_2} + \frac{\g_6}{\g_4} \frac{\g_4}{\g_2} + \ldots \Big)
\\ & 
< \frac{r}{2} \Big( 1 + \frac{r^2}{4} + \frac{r^2}{6} \frac{r^2}{4} + \ldots \Big)
= \frac{r}{2} \sum_{k=0}^\infty \frac{1}{(k+1)!} \Big( \frac{r^2}{2} \Big)^k
< \frac{r}{2} e^{r^2/2}.
\end{align*}
One has $\mL J = 1$ because $\mL 1 = 1$ and $\mL \eta = 0$. 
\end{proof}

We study the linear problem $\mL h = g$.

\begin{lemma} \label{lemma:mL inv} 
Let $r > 0$. 
For every $g \in X_r$ there exists a unique $h \in X_r^0$ such that $\mL h = g$. 
This defines the inverse operator 
$\mL^{-1} : X_r \to X_r^0$, $g \mapsto h = \mL^{-1} g$. 
Moreover 
\begin{equation} \label{stima mL inv}
\| \mL^{-1} g \|_r = \| h \|_r 
\leq r^2 e^{r^2/2} \| \mG g \|_r,
\end{equation}
where $\mG$ is the operator defined by 
\begin{equation} \label{def mG}
\mG g(x) = \sum_{n=0}^\infty \frac{g_n}{(n+2) \langle n \rangle} \, x^n. 
\end{equation}
\end{lemma}

\begin{proof}
Let $g(x) = \sum_{n=0}^\infty g_n x^n$, $h(x) = \sum_{n=2}^\infty h_n x^n$, 
with $g_n, h_n = 0$ for odd $n$. 
By~\eqref{mL f}, we have to solve 
\[
(n+2)^2 h_{n+2} - (n-1) h_n = g_n 
\quad \forall n \geq 0, \ \text{$n$ even}.
\] 
The solution is then recursively determined as 
\[
h_{n+2} = A_{n+2}^0 g_0 + \ldots + A_{n+2}^n g_n
\]
where 
\begin{equation} \label{ricorsiva Ank}
A_{n+2}^n = \frac{1}{(n+2)^2}, \quad \ 
A_{n+2}^k = \frac{n-1}{(n+2)^2} \, A_n^k \quad \ \forall k \in [0, n-2].
\end{equation}
We estimate
\begin{align}
\| h \|_r & = \sum_{n=2}^\infty |h_n| n r^{n-1}
= \sum_{n=2}^\infty \Big| \sum_{k=0}^{n-2} A_n^k g_k \Big| n r^{n-1} 
\leq \sum_{n=2}^\infty \sum_{k=0}^{n-2} A_n^k |g_k| n r^{n-1}
\notag \\ &
= \sum_{k=0}^\infty \Big( \sum_{n=k+2}^{\infty} A_n^k r^{n-k} \frac{n}{\langle k \rangle} \Big) 
|g_k| \langle k \rangle r^{k-1}
= \sum_{k=0}^\infty \Big( \sum_{j=2}^{\infty} \b_j(k,r) \Big) |g_k| \langle k \rangle r^{k-1}, 
\label{stima h}
\end{align}
where 
$\b_j(k,r) := A_{k+j}^k r^j (k+j) \langle k \rangle^{-1}$. 
By~\eqref{ricorsiva Ank}, 
\[
\frac{\b_{j+2}(k,r)}{\b_j(k,r)} 
= \frac{A_{k+j+2}^k r^{j+2} (k+j+2)}{A_{k+j}^k r^j (k+j)} 
= r^2 \frac{k+j-1}{(k+j+2)(k+j)} \to 0 \quad (j \to \infty),
\]
therefore the series $\sum_{j=2}^\infty \b_j(k,r)$ converges for all $k \geq 0$, all $r>0$.
By~\eqref{ricorsiva Ank}, 
\begin{align*}
\sum_{j=2}^\infty \b_j(k,r)
& = \b_2 + \b_2 \frac{\b_4}{\b_2} 
+ \b_2 \frac{\b_4}{\b_2} \frac{\b_6}{\b_4} + \ldots
\\ & 
\leq \frac{r^2}{(k+2)\langle k \rangle} \Big[ 1 + \frac{r^2}{k+4} + \frac{r^4}{(k+4)(k+6)} 
+ \frac{r^6}{(k+4)(k+6)(k+8)} + \ldots \Big]
\\ & 
\leq \frac{r^2}{(k+2)\langle k \rangle} \Big[ 1 + \frac{r^2}{4} + \frac{r^4}{6 \cdot 4}
+ \frac{r^6}{8 \cdot 6 \cdot 4} + \ldots \Big]
\\ & 
\leq \frac{r^2}{(k+2)\langle k \rangle} \Big( \sum_{n=0}^\infty \Big( \frac{r^2}{2} \Big)^n \frac{1}{n!} \Big)
= \frac{ r^2 e^{r^2 / 2}}{(k+2)\langle k \rangle}.
\end{align*}
Inserting this bound into~\eqref{stima h} gives~\eqref{stima mL inv}.
\end{proof}

Now we prove the following existence result for the nonlinear problem in analytic class.

\begin{lemma}
\label{lemma:LWP generico analytic}
Let $a, r, R, L$ be real positive numbers satisfying
\begin{align} \label{1311.4}
& e^{r^2/2} \Big[ \frac12 ar + \frac14 (1 + r^2) R^3 + \frac14 ar R^2 \Big] 
\leq R,
\\
& e^{r^2/2} \Big[ \frac34 (1 + r^2) R^2 + \frac12 ar R \Big] 
\leq L < 1.
\label{1311.5}
\end{align}
Then there exists a unique analytic function $h \in X_r^0$, $\| h \|_r \leq R$, 
that solves the Cauchy problem~\eqref{Cp h} in the interval $[-r,r]$.
\end{lemma}

\begin{proof}
By Lemma~\ref{lemma:mL inv}, the equation $\mL h = Q(h,a)$ 
can be written as the fixed point problem in $X_r^0$
\begin{equation} \label{nonlin mL analytic}
h = \Psi(h) 
\quad \ \text{where} \ \Psi(h) := \Psi(h,a) := \mL^{-1} Q(h,a).
\end{equation}
We prove that, for suitable $r,R$, 
the map $\Psi$ is a contraction in the ball 
\begin{equation} \label{ball analytic}
B_R(X_r^0) := \{ h \in X_r^0 : \| h \|_r \leq R \}.
\end{equation}
To this aim, we estimate separately each term of $\Psi$: 
\begin{alignat*}{2}
& \Psi_1 := \mL^{-1}(-a) = - a \mL^{-1} (1), 
\qquad & 
& \Psi_4 := \mL^{-1} (-h'^2 h), 
\\ 
& \Psi_2 := \mL^{-1} (x h'^3), 
\quad &
& \Psi_5 := \mL^{-1}(-a h'^2), 
\\ 
& \Psi_3 := \mL^{-1} (- \tfrac{1}{x} h'^3), 
\qquad & 
& \Psi = \Psi_1 + \Psi_2 + \Psi_3 + \Psi_4 + \Psi_5.
\end{alignat*}

\noindent
\emph{Estimate of $\Psi_1$}. 
By Lemma~\ref{lemma:kernel}, $\mL^{-1}(1) = J$, and 
$\| \Psi_1 \|_r = a \| J \|_r \leq \frac12 a r e^{r^2 / 2}$.

\noindent
\emph{Estimate of $\Psi_3$}.
Let $h(x) = \sum_{n=2}^\infty c_n x^n$. Then 
\[ 
\frac{1}{x} h'^3(x)
= \sum_{k_1, k_2, k_3 \geq 2} 
c_{k_1} c_{k_2} c_{k_3} \, 
k_1 k_2 k_3 \,
x^{k_1 + k_2 + k_3 - 4}
\]
and 
\[
\mG \Big( \frac{1}{x} h'^3 \Big)(x)
= \sum_{n=2}^\infty \Big( \sum_{\begin{subarray}{c} k_1, k_2, k_3 \geq 2 \\ 
k_1 + k_2 + k_3 - 4 = n \end{subarray}} 
\frac{c_{k_1} c_{k_2} c_{k_3} \, k_1 k_2 k_3}{(n+2) n} \Big) x^n
\]
where $\mG$ is defined in~\eqref{def mG}. 
By~\eqref{stima mL inv}, 
\begin{align*}
\| \Psi_3 \|_r 
& = \| \mL^{-1}( \tfrac{1}{x} h'^3) \|_r 
\leq r^2 e^{r^2/2} \| \mG( \tfrac{1}{x} h'^3) \|_r 
\\ & 
= r^2 e^{r^2/2} 
\sum_{n = 2}^\infty \Big| \sum_{\begin{subarray}{c} k_1, k_2, k_3 \geq 2 \\ 
k_1 + k_2 + k_3 - 4 = n \end{subarray}} 
\frac{c_{k_1} c_{k_2} c_{k_3} \, k_1 k_2 k_3}{(n+2)n} \Big| n r^{n-1}
\\ & 
\leq r^2 e^{r^2/2} 
\sum_{n = 2}^\infty \sum_{\begin{subarray}{c} k_1, k_2, k_3 \geq 2 \\ 
k_1 + k_2 + k_3 - 4 = n \end{subarray}} 
\frac{|c_{k_1}| |c_{k_2}| |c_{k_3}| \, k_1 k_2 k_3}{4} r^{n-1}
\\ & 
= \frac14 r^2 e^{r^2/2} 
\sum_{k_1, k_2, k_3 \geq 2} 
|c_{k_1}| |c_{k_2}| |c_{k_3}| \, k_1 k_2 k_3 \, r^{k_1 + k_2 + k_3 - 5}
\\ & 
\leq \frac14 r^2 e^{r^2/2} 
\Big( \sum_{k_1 \geq 2} |c_{k_1}| k_1 r^{k_1 - 1} \Big) 
\Big( \sum_{k_2 \geq 2} |c_{k_2}| k_2 r^{k_2 - 1} \Big) 
\Big( \sum_{k_3 \geq 2} |c_{k_3}| k_3 r^{k_3 - 1} \Big) r^{-2}
\\ & 
= \frac14 e^{r^2/2} \| h \|_r^3.
\end{align*}

\noindent
\emph{Estimate of $\Psi_2$}. One has 
\[
\mG (x h'^3)(x)
= \sum_{n=4}^\infty \Big( \sum_{\begin{subarray}{c} k_1, k_2, k_3 \geq 2 \\ 
k_1 + k_2 + k_3 - 2 = n \end{subarray}} 
\frac{c_{k_1} c_{k_2} c_{k_3} \, k_1 k_2 k_3}{(n+2) n} \Big) x^n
\]
and therefore, proceeding as for $\Psi_3$, 
we obtain $\| \Psi_2 \|_r \leq \tfrac16 r^2 e^{r^2/2} \| h \|_r^3$.

\noindent
\emph{Estimate of $\Psi_4$}. One has 
\[
\mG(h'^2 h)(x) 
= \sum_{n = 4}^\infty \Big( \sum_{\begin{subarray}{c} k_1, k_2, k_3 \geq 2 \\ 
k_1 + k_2 + k_3 - 2 = n \end{subarray}} 
\frac{c_{k_1} c_{k_2} c_{k_3} \, k_1 k_2}{(n+2) n} \Big) x^n.
\]
Since $|c_{k_3}| \leq \frac12 |c_{k_3}| k_3$, 
proceeding as above we get 
$\| \Psi_4 \|_r$ $\leq \tfrac{1}{12} r^2 e^{r^2/2} \| h \|_r^3$.

\noindent
\emph{Estimate of $\Psi_5$}. One has 
\[
\mG(h'^2)(x) 
= \sum_{n=2}^\infty \Big( \sum_{\begin{subarray}{c} k_1, k_2 \geq 2 \\ 
k_1 + k_2 - 2 = n \end{subarray}} 
\frac{c_{k_1} c_{k_2} \, k_1 k_2}{(n+2)n} \Big) x^n.
\]
Proceeding as above, we get 
$\| \Psi_5 \|_r 
\leq \tfrac{1}{4} a r e^{r^2/2} \| h \|_r^2$.

\noindent
\emph{Estimate of $\Psi$}. Collecting the estimates above, we have proved that 
\[
\| \Psi(h) \|_r 
\leq e^{r^2/2} \big( \tfrac12 a r + \tfrac14 a r \| h \|_r^2 
+ \tfrac14 (1 + r^2) \| h \|_r^3 \big).
\]
Similarly, since $\Psi_1$ cancels in the difference $\Psi(h) - \Psi(g)$, 
one proves that 
\begin{align} 
\| \Psi(h) - \Psi(g) \|_r 
& \leq e^{r^2/2} 
\Big\{ \frac14 (1 + r^2) \Big( \| h \|_r^2 + \| h \|_r \| g \|_r 
+ \| g \|_r^2 \Big)
\notag \\ & \qquad 
+ \frac14 a r \Big( \| h \|_r + \| g \|_r \Big) \Big\}
\| h - g \|_r.
\label{1311.1}
\end{align}

\noindent
\emph{Contraction}. 
For $h,g$ in the ball $B_R(X_r^0) = \{ h \in X_r^0 : \| h \|_r \leq R \}$ one has 
\begin{align} 
\| \Psi(h) - \Psi(g) \|_r 
& \leq e^{r^2/2} \Big( \frac34 (1 + r^2) R^2 + \frac12 a r R \Big) \| h - g \|_r,
\label{1311.2}
\\ 
\| \Psi(h) \|_r 
& \leq e^{r^2/2} \Big( \frac12 a r + \frac14 a r R^2 
+ \frac14 (1 + r^2) R^3 \Big).
\label{1311.15}
\end{align}
Assumptions~\eqref{1311.4}-\eqref{1311.5} 
imply that $\Psi$ is a contraction of the ball $B_R(X_r^0)$ into itself.
\end{proof}

\begin{lemma} \label{lemma:f+aJ}
The solution $h$ in Theorem~\ref{lemma:LWP generico analytic} satisfies 
\[
\| h + a J \|_r 
\leq \frac{L a r e^{r^2/2}}{2(1 - L)}\,.
\]
\end{lemma}

\begin{proof}
Let $h_0 = 0$, $h_{n+1} = \Psi(h_n)$. 
Then $h_1 = \Psi(0) = - a J$, 
\begin{align*}
h - h_1 & = \sum_{n=1}^\infty (h_{n+1} - h_n)
= \sum_{n=1}^\infty (\Psi(h_n) - \Psi(h_{n-1})),
\\
\| h - h_1 \|_r 
& \leq \sum_{n=1}^\infty \| \Psi(h_n) - \Psi(h_{n-1}) \|_r 
\leq \sum_{n=1}^\infty L^n \| h_1 \|_r 
= \frac{L}{1-L} a \| J \|_r,
\end{align*}
and $\| J \|_r \leq \frac12 r e^{r^2/2}$ by Lemma~\ref{lemma:kernel}.
\end{proof}

\begin{lemma} \label{lemma:regime stellina}
Let $r,a,R,L$ be positive real numbers, with $L < 1$. 
If 
\begin{equation} \label{1411.1}
R \leq C_r \sqrt{L}, \quad 
a \leq K_r R, 
\end{equation}
where
\begin{equation} \label{1411.2}
C_r := \frac{\sqrt{2}}{e^{r^2/4} \sqrt{3(1 + r^2)}}\,, \quad 
K_r := \frac{1}{r e^{r^2/2}}\,,
\end{equation}
then~\eqref{1311.4}-\eqref{1311.5} hold.
\end{lemma}

\begin{proof}
It is a straightforward check.
\end{proof}

We recall (see Lemma~\ref{lemma:kernel}) that all the derivatives 
of the function $J$ are strictly positive, 
and $J(0) = 0$, $J'(0) = 0$. 

\begin{definition} \label{def x0}
We denote by $x_0$ the positive real number such that $J(x_0) = 1$. 
\end{definition}

\begin{remark} \label{rem:long time a small}
The solution $h$ of the Cauchy problem~\eqref{Cp h} 
given by Lemma~\ref{lemma:LWP generico analytic}
is defined on an interval $[- r, r]$, with $r \to \infty$ as $a \to 0$. 
To see this, apply Lemma~\ref{lemma:regime stellina} with 
\[
L = \frac12, \quad 
R = \frac{1}{e^{r^2/4} \sqrt{3 (1+r^2)}}, \quad 
a = \frac{1}{e^{3r^2/4} r \sqrt{3 (1+r^2)}}.
\]
Hence the interval $[-r,r]$ contains $[-2 x_0, 2 x_0]$ for all $a$ sufficiently small. 
\end{remark}

\begin{lemma} \label{va a zero}
There exists a universal constant $a_* > 0$ 
such that, for every $a \in (0,a_*)$, 
there exists a unique $\xi_a \in (0,2 x_0)$
such that the solution $h_a$ of the Cauchy problem~\ref{Cp h} 
satisfies $h_a(\xi_a) = - a$. 
Moreover $h_a'(\xi_a) \to 0$ as $a \to 0$.

As a consequence, the solution $f_a(x) = a + h_a(x)$ 
of equation~\eqref{1010.1} with initial data $f_a(0) = a$, $f_a'(0) = 0$ satisfies 
\[
f_a(\xi_a) = 0, \quad 
f_a'(\xi_a) \to 0 \quad \text{as} \ a \to 0.
\]
\end{lemma}

\begin{proof}
Let $r = 2 x_0$, and let $C_r, K_r$ be given by~\eqref{1411.2}.
Let $a_0 := C_r K_r$. 
For every $a < a_0$ we fix 
$R = a/K_r$ and $L = R^2 / C_r^2 = a^2 / a_0^2$. 
Hence, by Lemma~\ref{lemma:regime stellina},~\eqref{1311.4}-\eqref{1311.5} hold,
and therefore, by Lemma~\ref{lemma:LWP generico analytic}, there exists a unique analytic solution 
$h_a \in X_r^0$, $\| h_a \|_r \leq R$, of the Cauchy problem~\ref{Cp h}. 
By Lemma~\ref{lemma:f+aJ}, for all $a \leq a_0/2$ one has
\begin{equation} \label{ha aJ cubo}
\| h_a + a J \|_r 
\leq \frac{Lar e^{r^2/2}}{2(1 - L)}
\leq \frac23 L a r e^{r^2/2}
\leq C a^3
\end{equation}
where $C > 0$ is a universal constant 
(since $x_0$ is a universal constant, 
also $r, C_r, K_r, a_0$ are universal). 

For every function $g(x) = \sum_{n=2}^\infty c_n x^n \in X_r^0$, 
since $g(0) = 0$, one has the uniform bound 
\[
|g'(x)| \leq \| g \|_r, \quad 
|g(x)| \leq r \| g \|_r \quad \forall |x| \leq r.
\]
Hence, by~\eqref{ha aJ cubo}, for all $a \in (0, a_0/2)$ we get
\[
|h_a'(x) + a J'(x)| \leq C a^3, \quad 
|h_a(x) + a J(x)| \leq r C a^3 \quad \forall |x| \leq r
\]
and, dividing by $a$, 
\[
\Big| \frac{h_a'(x)}{a} + J'(x) \Big| \leq C a^2, \quad 
\Big| \frac{h_a(x)}{a} + J(x) \Big| \leq r C a^2 \quad \forall |x| \leq r.
\]
Thus the function $h_a/a$ converges to $-J$ in $C^1([-r,r])$ as $a \to 0$. 
As a consequence, since $J(x_0) = 1$ and $J'(x_0) > 0$, 
for every $a$ small enough there exists a unique 
$\xi_a \in [0,r]$ such that $h_a(\xi_a) / a = -1$. 
Moreover $\xi_a \to x_0$ 
and $h_a'(\xi_a) / a \to - J'(x_0)$ as $a \to 0$. 
\end{proof}

\section{Extension of local solutions that are graphs} 
\label{sec:extension graph}

In this section we extend the solutions $f$ beyond the local existence of the previous sections, 
without assuming that $a$ is small.

\begin{lemma}\label{lemma3.1}
Let $a > 0$, $r \in (0,1)$, 
and let $f \in C^2([0,r])$ solve 
\begin{equation} \label{1710.1}
f''(x) = \big( 1 + f'^2(x) \big) \Big[ f'(x) \Big( x - \frac{1}{x} \Big) - f(x) \Big]
\quad \forall x \in (0,r],
\end{equation}
\begin{equation} \label{1710.2}
f(0) = a, \quad f'(0) = 0.
\end{equation}
Assume that 
\begin{equation} \label{assum.ff}
f''(x) < 0, \quad 
f(x) > 0 \quad 
\forall x \in [0,r].
\end{equation}
Then, for every $x \in (0,r]$,
\begin{equation}\label{bellestime}
a \sqrt{1 - x^2} < f(x) < a, 
\qquad 
- \frac{a x}{1 - x^2} < f'(x) < 0. 
\end{equation}
\end{lemma}

\begin{proof} 
The hypothesis $f''(x) < 0$ 
and~\eqref{1710.1} imply that 
\begin{equation} \label{1710.3}
f'(x) \Big( x - \frac{1}{x} \Big) - f(x) < 0 \quad \forall x \in (0,r].
\end{equation}
By assumption, $f(x) > 0$ and $(x - \frac{1}{x}) < 0$ for $x \in (0,r] \subset (0,1)$. 
Then~\eqref{1710.3} gives
\[
\frac{f'(x)}{f(x)} > \frac{1}{x - \frac{1}{x}} = \frac{x}{x^2 - 1}
= \frac12 \cdot \frac{-2x}{1 - x^2},
\]
namely
\[
\frac{d}{dx} \log(f(x)) > \frac{d}{dx} \log \sqrt{1 - x^2}.
\]
Hence the function 
\begin{equation}\label{Lyapu}
F(x) := \frac{f(x)}{\sqrt{1-x^2}}
\end{equation}
is strictly increasing in $[0,r]$. 
Since $F(0) = a$, we deduce the first inequality in~\eqref{bellestime}.
The bounds $f'(x) < 0$ and $f(x) < a$ are a direct consequence of the hypothesis $f''(x) < 0$ 
and~\eqref{1710.2}.
By~\eqref{1710.3}, since $f(x)<a$ for $x \in (0,r]$, one has
\[
f'(x) \Big( x - \frac1x \Big) < f(x) < a, \quad 
f'(x) > \frac{a}{x - \frac1x} = \frac{ax}{x^2 - 1} = - \frac{ax}{1 - x^2}.
\]
The proof of~\eqref{bellestime} is complete.
\end{proof}

\begin{lemma}\label{lemma concave}
Let $a > 0$, $r \in (0,1)$, 
and let $f \in C^2([0,r])$ solve~\eqref{1710.1}-\eqref{1710.2}. 
Then $f(x) > 0$ for all $x \in [0,r]$ and $f''(x)<0$ for all $x \in [0,r)$.
\end{lemma}

\begin{proof}
Let 
\[
E := \{ x^* \in [0,r] : f(x) > 0, \ f''(x) < 0 \ \forall x \in [0, x^*] \}, 
\quad p := \sup E.
\]
By~\eqref{1710.1}-\eqref{1710.2}, one has $f''(0) = - a/2 < 0$. 
Then $f(x) > 0$, $f''(x) < 0$ for all $x \in [0,\d]$, for some $\d > 0$. 
Therefore $0 < \d \leq p \leq r$. 
Assume, by contradiction, that $p < r$. 
For all $x \in [0,p)$, one has $f(x) > 0$ and $f''(x) < 0$. 
By Lemma~\ref{lemma3.1}, $f(p) \geq a \sqrt{1-p^2} > 0$. Hence $f''(p) = 0$. 
The function $F(x)$ in~\eqref{Lyapu} is strictly increasing in $[0,p]$, 
as in the proof of Lemma~\ref{lemma3.1}.
Hence $F(x) < F(p)$ for all $x \in [0,p)$. 
Let 
\begin{equation}\label{gpiccola}
g(x) := f(p) \sqrt{\frac{1-x^2}{1-p^2}} = F(p) \sqrt{1-x^2}.
\end{equation}
By construction, $f(p) = g(p)$, and
\begin{equation} \label{f<g}
f(x) = F(x) \sqrt{1-x^2} 
< F(p) \sqrt{1-x^2} = g(x) \qquad \forall x \in [0,p).
\end{equation}
Moreover, 
$$
g'(x) = - \frac{x F(p)}{\sqrt{1-x^2}}, 
\qquad g'(p) 
= - \frac{p F(p)}{\sqrt{1-p^2}}
= - \frac{p f(p)}{1-p^2}.
$$
Since $f''(p) = 0$, by~\eqref{1710.1} we deduce that 
$$
f'(p) = - \frac{p f(p)}{1-p^2} = g'(p).
$$
Moreover we calculate $g''(p) = - F(p) (1-p^2)^{-3/2} < 0 = f''(p)$. 
Summarizing, 
$$
f(p)=g(p), \quad 
f'(p)=g'(p), \quad 
f''(p) > g''(p),
$$
which contradicts~\eqref{f<g}. 
This proves that $p=r$. 
Hence $f(x)>0$, $f''(x)<0$ for all $x \in [0,r)$. 
By Lemma~\ref{lemma3.1} and continuity, 
$f(r) \geq a \sqrt{1-r^2} > 0$.
\end{proof}

\begin{coro} \label{cor:coro}
Let $a>0$. Then there exists a unique solution $f \in C^2([0,1))$ of~\eqref{1710.1}-\eqref{1710.2}. It satisfies~\eqref{bellestime} and $f''(x)<0$ for all $x \in (0,1)$. 
\end{coro}

\begin{proof}
It follows from Lemmas~\ref{lemma3.1} and~\ref{lemma concave}.
\end{proof}

\begin{lemma} \label{bel lemma}
Let $f \in C^2([0,1))$ be the solution in Corollary~\ref{cor:coro}. 
Then 
\[
0 < \lim_{x \to 1} f(x) < a, \qquad 
- \infty < \lim_{x \to 1} f'(x) < 0, \qquad 
- \infty < \lim_{x \to 1} f''(x) < 0.
\]
\end{lemma}

\begin{proof}
Since $f$ and $f'$ are both decreasing, 
the limits $b := \lim f(x)$, $c := \lim f'(x)$ as $x \to 1$ exist. 
Assume, by contradiction, that $c = -\infty$. 
Then system~\eqref{geo.5} for $(u,v)(s)$ with initial data 
\[
(u,v)(s_0) = (1,b), \quad 
(u',v')(s_0) = (0,1)
\]
has two solutions: the one that runs along the vertical line $x = 1$, namely
\[
(u,v)(s) = (1, b + s - s_0) \quad \forall s \in \R,
\]
and the one that runs along the curve $y = f(x)$, which is a contradiction. 

Thus $c$ is finite. Assume that $b = 0$. 
From the bound $a \sqrt{1 - x^2} < f(x)$ in~\eqref{bellestime} it follows that 
$c = -\infty$ (the graph of $f$ must have vertical tangent at the point $(1,0)$), 
and this is a contradiction. Then $b > 0$. 

From~\eqref{1710.1} we deduce that $\lim_{(x \to 1)} f''(x) = - (1+c^2) b$, 
a finite negative number. 
\end{proof}

\begin{corollary}\label{bel coro}
Let $a > 0$. Then there exist $\d > 0$ and a unique solution $f \in C^2([0,1+\d])$ 
of~\eqref{1710.1}-\eqref{1710.2}.
The solution satisfies~\eqref{bellestime} on $(0,1)$ and 
\begin{equation} \label{1+}
f(x) > 0, \quad 
f'(x) < 0, \quad 
f''(x) < 0 \quad 
\forall x \in (0,1+\d].
\end{equation} 
\end{corollary}

\begin{proof}
By Lemma~\ref{bel lemma} and standard local existence and uniqueness for Cauchy problems, 
the solution in Corollary~\ref{cor:coro} can be extended beyond $x=1$.
\end{proof}

In the next lemma, by~\eqref{1+}, we prove an estimate 
that will be used in Section~\ref{sec:TG}. 

\begin{lemma} \label{lemma:testa nel sacco}
Let $a > 0$, and let $f$ be the solution of~\eqref{1710.1}-\eqref{1710.2} 
in Corollary~\ref{bel coro}. Then 
\begin{equation} \label{f gira}
\frac{f(x) - x f'(x)}{\sqrt{1 + f'^2(x)}} \geq \frac{a}{\sqrt{1+a^2}} 
\quad \forall x \in [0,1].
\end{equation}
\end{lemma}

\begin{proof}
Since $f''<0$, one has 
\[
\frac{f(x) - f(1)}{x-1} < f'(x) < \frac{f(x) - f(0)}{x} \quad \forall x \in (0,1).
\]
Since $f(1) > 0$ and $f(0) = a$, 
\begin{equation} \label{bound f'}
- \frac{f(x)}{1-x} < f'(x) < - \frac{a - f(x)}{x}.
\end{equation}
We write the left-hand side term of~\eqref{f gira} as $H(\ph_0)$, 
where $H : (- \tfrac{\pi}{2}, 0) \to \R$, 
\[
H(\ph) := f(x) \cos (\ph) - x \sin(\ph), 
\quad 
\ph_0 := \arcsin (\psi(f'(x)) \in (- \tfrac{\pi}{2}, 0), \quad 
\psi(t) := \frac{t}{\sqrt{1+t^2}}. 
\]
Since $\psi$ and $\arcsin$ are increasing,~\eqref{bound f'} gives
\[
\ph_1 := \arcsin\Big( \psi \Big( - \frac{f(x)}{1-x} \Big) \Big)
< \ph_0 
< \arcsin\Big( \psi \Big( - \frac{a - f(x)}{x} \Big) \Big) =: \ph_2,
\]
and $\ph_1, \ph_2 \in (-\frac{\pi}{2},0)$. 
One has $H(\ph) > 0$ and $H''(\ph) = - H(\ph) < 0$ for all $\ph \in (-\frac{\pi}{2},0)$. 
Hence $H(\ph_0) > \min \{ H(\ph_1), H(\ph_2) \}$ because $\ph_0 \in (\ph_1, \ph_2)$. 
We calculate 
\[
H(\ph_1) = \frac{f(x)}{\sqrt{(1-x)^2 + f^2(x)}}\,, \quad 
H(\ph_2) = \frac{ax}{\sqrt{x^2 + (a - f(x))^2}}\,.
\]
Since $(1-x) < f(x) / a$, one has $H(\ph_1) > a / \sqrt{1+a^2}$ and,
since $(a - f(x)) < ax$, also $H(\ph_2) > a / \sqrt{1+a^2}$.
\end{proof}

\section{Intersection of the curve with the horizontal axis} 
\label{sec:TG}

In this section we go back to the problem~\eqref{geo.2}-\eqref{geo.4} (or equivalently~\eqref{geo.5}) 
for the curve $(u(s), v(s))$, and we follow the curve proving that it reaches the horizontal axis $v=0$, possibly extending it beyond the portion where it is a graph $v = f(u)$.
Recall (definition~\eqref{geo.3}) that the curvature of the curve $(u,v)$ is
\begin{equation} \label{u.0}
k = - v' u'' + u' v'',
\end{equation}
where $'$ denotes the derivative with respect to the arclength $s$. 

\begin{lemma} \label{lemma:u'' v'' k}
Let $a>0$, and let $f \in C^2([0, 1+\d])$ be the solution of~\eqref{1710.1}-\eqref{1710.2}
of Corollary~\ref{bel coro}. Then the corresponding functions $u(s), v(s)$ defined by 
\eqref{geo.6}-\eqref{geo.8} belong to $C^2([0, s_1])$, 
where $s_1 = \int_0^{1+\d} \sqrt{1+f'^2(x)} \, dx$ $> 1$, 
they solve~\eqref{geo.5} on $s \in (0,s_1]$ with initial data 
\begin{equation} \label{u.6}
u(0) = 0, \quad 
v(0) = a, \quad 
u'(0) = 1, \quad 
v'(0) = 0,
\end{equation}
and
\begin{equation} \label{u.7}
u''(0) = 0, \quad 
v''(0) = - \frac{a}{2}, \quad 
k(0) = - \frac{a}{2}\,.
\end{equation}
\end{lemma}

\begin{proof} 
It is a direct consequence of~\eqref{geo.6}-\eqref{geo.8} and Corollary~\ref{bel coro}.
\end{proof}

We prove that, along any solution of the Cauchy problem~\eqref{geo.5}-\eqref{u.6}, 
the curvature $k$ satisfies an integral formula.

\begin{lemma} \label{lemma:var of const}
Assume that $(u(s), v(s))$ solves the Cauchy problem~\eqref{geo.5}-\eqref{u.6} 
on some interval $s \in (0,s_0]$. Then for all $s \in (0,s_0]$ the curvature 
$k$ (defined in~\eqref{u.0}) satisfies 
\begin{equation} \label{u.12}
k(s) = \ph(s) e^{\frac12 (u^2(s) + v^2(s))} \frac{1}{u(s)}, \qquad 
\ph(s) := \int_0^s e^{- \frac12 (u^2(t) + v^2(t))} \frac{u'(t) v'(t)}{u(t)} \, dt.
\end{equation}
\end{lemma}

\begin{proof}
If $(u,v)$ solves equation~\eqref{geo.4}, then 
\begin{equation} \label{u.1}
k = - \frac{ v' }{ u } + u v' - v u'.
\end{equation}
Since $u'^2 + v'^2 = 1$ (arclength), one has $\g'' = k \nu$, namely
$u'' = - v' k$, $v'' = u' k$. 
Substituting in~\eqref{u.1} gives
\begin{equation} \label{u.2}
k' + \Big( \frac{u'}{u} - u u' - v v' \Big) k = \frac{u' v'}{u^2}\,.
\end{equation}
The factor in parenthesis is the derivative of $\log(u) - \frac12(u^2 + v^2)$.
Thus, by variation of constants, we can write $k$ in the form
\begin{equation} \label{u.3}
k = (c + \ph) e^{\frac12 (u^2 + v^2)} \frac{1}{u}
\end{equation}
where $c$ is a constant and $\ph(s)$ satisfies 
\begin{equation} \label{u.4}
\ph' = e^{- \frac12 (u^2 + v^2)} \frac{u' v'}{u}\,.
\end{equation}
By Lemma~\ref{lemma:u'' v'' k} and de L'H\^opital's rule, 
the right hand side of~\eqref{u.4} 
tends to $ - \frac{a}{2} \, e^{- \frac12 a^2}$ as $s \to 0$. 
Hence $\ph'$ has a removable singularity at $s = 0$, 
and we can define 
\[ 
\ph(s) := \int_0^s \Big( e^{- \frac12 (u^2 + v^2)} \frac{u' v'}{u} \Big)(t) \, dt.
\] 
As a consequence, the limit as $s \to 0$ in~\eqref{u.3} gives $c = 0$. 
\end{proof}

Now we elaborate on $\ph$ and we find another formula for the curvature.

\begin{lemma} \label{lemma:intercono}
Assume the same hypotheses as in Lemma~\ref{lemma:var of const}. 
Then the curvature $k$ satisfies 
\begin{equation} \label{u.17}
k(s) = - \frac{v'(s)}{u(s)} - 
\frac{ e^{\frac12 (u^2 + v^2)(s)} }{u(s)} 
\int_0^s e^{- \frac12 (u^2 + v^2)} v \, dt.
\end{equation}
\end{lemma}

\begin{proof}
Let $\ph(s)$ be the integral defined in~\eqref{u.12}.
Recalling that $v'(0) = 0$, integration by parts yields
\begin{align*}
\ph(s) & = \int_0^s u' \, \Big( e^{- \frac12 (u^2 + v^2)} \, \frac{v'}{u} \Big) \, dt 
= \Big[ v' e^{- \frac12 (u^2 + v^2)} \Big]_{t = 0}^{t = s} 
- \int_0^s u \Big( \frac{v'}{u} \, e^{- \frac12 (u^2 + v^2)} \Big)' \, dt
\\ 
& = v'(s) e^{- \frac12 (u^2 + v^2)(s)} 
- \int_0^s e^{- \frac12 (u^2 + v^2)} \Big( v'' - \frac{v' u'}{u} - v' u u' - v'^2 v \Big) \, dt.
\end{align*}
Using the equation for $v''$ in~\eqref{geo.5}, and then the identity $u'^2 + v'^2 = 1$, 
\[
- \Big( v'' - \frac{v' u'}{u} - v' u u' - v'^2 v \Big)
= 2 \frac{v' u'}{u} + v.
\]
Hence 
\begin{align*}
\ph(s) 
& = v'(s) e^{- \frac12 (u^2 + v^2)(s)} 
+ 2 \ph(s) + \int_0^s e^{- \frac12 (u^2 + v^2)} v \, dt.
\end{align*}
Therefore 
\begin{equation} \label{u.16}
\ph(s) = - v'(s) e^{- \frac12 (u^2 + v^2)(s)} 
- \int_0^s e^{- \frac12 (u^2 + v^2)} v \, dt.
\end{equation} 
Replacing~\eqref{u.16} into the first identity in~\eqref{u.12} gives~\eqref{u.17}.
\end{proof}

Now we use formula~\eqref{u.17} to study the solution beyond its first intersection 
with the vertical line $u=1$. 

\begin{definition} \label{def:s star}
Let $a > 0$, and let $(u,v)$ be the solution of~\eqref{geo.5}-\eqref{u.6} 
of Lemma~\ref{lemma:u'' v'' k}. 
We denote by $s_* = s_*(a) > 1$ the smallest positive $s$ such that $u(s) = 1$. 
\end{definition}

\begin{lemma} \label{lemma:gira gira}
Assume the same hypotheses as in Lemma~\ref{lemma:var of const}. 
Also assume that $s_0 > s_*$ and that $u(s) > 0$, $v(s) \geq 0$ 
for all $s \in [s_*,s_0]$, with $s_*$ as in Definition~\ref{def:s star}. 
Then 
\begin{equation} \label{u.18}
\frac{\g}{|\g|} \cdot \nu = \frac{- u v' + v u'}{\sqrt{u^2 + v^2}} 
> \frac{\pi \sqrt{e}}{8} \, a e^{-\frac12 a^2} \quad \forall s \in [s_*,s_0].
\end{equation}
\end{lemma}

\begin{proof}
By~\eqref{u.1} and~\eqref{u.17} we get
\begin{equation}\label{u.19}
\frac{- u v' + v u'}{\sqrt{u^2 + v^2}}(s) = 
\frac{ e^{\frac12 (u^2 + v^2)}}{u\sqrt{u^2 + v^2}} (s)
\int_0^s e^{- \frac12 (u^2 + v^2)} v \, dt 
= (T_1 T_2 T_3)(s),
\end{equation}
where
\begin{equation}\label{def T123}
T_1(s) := \frac{ e^{\frac14 (u^2 + v^2)}}{u} (s), \quad
T_2(s) := \frac{ e^{\frac14 (u^2 + v^2)}}{\sqrt{u^2 + v^2}}(s), \quad
T_3(s) := \int_0^s e^{- \frac12 (u^2 + v^2)} v \, dt.
\end{equation}
We give a lower bound for these three terms, separately. 
Using the fact that $t^{-1}e^{t^2/4} \geq e^\frac12/\sqrt{2}$ for all $t>0$, we have
\begin{equation}\label{lower bound T12}
T_1 (s) \geq \frac{ e^{\frac14 u^2}}{u} (s) \geq \frac{e^\frac12}{\sqrt{2}}, \qquad
T_2 (s) \geq \frac{e^\frac12}{\sqrt{2}} \qquad \forall s\in [s_*,s_0].
\end{equation}
To give a lower bound for $T_3$, we recall (see Corollary~\ref{bel coro}) that on $[0,s_*]$ 
the curve $(u(s),v(s))$ is the graph of a function $y=f(x)$, $x \in [0,1]$, satisfying~\eqref{bellestime}. 
Then 
\begin{align} 
\int_0^{s_*} v(t) \, dt = \int_0^1 f(x) \sqrt{1+(f'(x))^2}\, dx
\geq \int_0^1 f(x) \, dx > a \int_0^1 \sqrt{1-x^2} \, dx = \frac{a\pi}{4}.
\label{stima int v}
\end{align}
Since $v \geq 0$ on $[s_*,s_0]$ and
$u \in [0,1]$, $v \in [0,a]$ on $[0, s_*]$, we deduce that
\begin{equation}\label{lower bound T3}
T_3 (s) \geq T_3(s_*)
\geq e^{- \frac12 (1 + a^2)} \int_0^{s_*} v(t) \, dt 
> e^{- \frac12 (1 + a^2)} \frac{a \pi}{4}
\qquad \forall s\in [s_*,s_0]
\end{equation}
and the thesis follows, recalling~\eqref{u.19}-\eqref{lower bound T12}.
\end{proof}

A similar estimate holds on $[0,s_*]$:

\begin{lemma} \label{lemma:gira gira [0,1]}
Let $a, u, v, s_*$ be as in Definition~\ref{def:s star}.
Then
\begin{equation} \label{u.18.bis}
\frac{- u v' + v u'}{\sqrt{u^2 + v^2}} 
\geq \frac{a}{1+a^2} \quad \forall s \in [0,s_*].
\end{equation}
\end{lemma}

\begin{proof} 
For $s \in [0,s_*]$, the curve $(u,v)(s)$ is the graph $v(s) = f(u(s))$, 
with $u \in [0,1]$, $v \in [0,a]$. Hence 
\[
\frac{- u v' + v u'}{\sqrt{u^2 + v^2}}
= \frac{1}{\sqrt{u^2 + f^2(u)}} \, \frac{f(u) - u f'(u)}{\sqrt{1 + f'^2(u)}}\,.
\]
The thesis follows from~\eqref{f gira} and $u^2 + f^2(u) \leq 1 + a^2$.
\end{proof}

Collecting~\eqref{u.18} and~\eqref{u.18.bis} we get the following lower bound
along the curve.

\begin{corollary} \label{cor:gira glob}
Assume the hypotheses of Lemma~\ref{lemma:gira gira}. 
Then 
\begin{equation} \label{u.18.glob}
\frac{- u v' + v u'}{\sqrt{u^2 + v^2}} > K_a 
\quad \forall s \in [0,s_0],
\qquad K_a := \frac{\pi \sqrt{e}}{8} a e^{-\frac12 a^2}.
\end{equation}
\end{corollary}

\begin{proof}
One has $a / (1+a^2) > K_a$ 
because the function $\chi(a) := e^{a^2/2} / (1+a^2)$ has its minimum at $a=1$, 
with $\chi(1) = \sqrt{e} / 2 > \pi \sqrt{e} / 8$.
\end{proof}

\medskip

We introduce the polar coordinates $u=\rho\cos\th$, $v=\rho\sin\th$, 
$\th \in [- \frac{\pi}{2}, \frac{\pi}{2}]$, 
on the half-plane $u \geq 0$.
The arclength parametri\-zation becomes 
\[
\rho'^2+\rho^2 \th'^2 = 1.
\] 
In these coordinates one has 
$(- u v' + v u') / \sqrt{u^2 + v^2} = - \rho \th'$,
and bound~\eqref{u.18.glob} gives 
\begin{equation} \label{rho theta cone}
(\rho^2 \th'^2) (s) > K_a^2, \qquad 
\rho'^2 (s) < 1- K_a^2 \qquad 
\forall s \in [0, s_0].
\end{equation}

\begin{lemma} \label{lemma:pallina}
Assume the hypotheses of Lemma~\ref{lemma:gira gira}. 
Then
\begin{equation}\label{rho away from zero}
e^{- C_a \pi / 2} < \rho(s) < e^{C_a \pi / 2} \quad \forall s \in [0, s_0],
\end{equation}
where $C_a := \sqrt{1-K_a^2} / K_a$ and $K_a$ is defined in~\eqref{u.18.glob}.
\end{lemma}

\begin{proof}
By~\eqref{u.6}, at $s=0$ one has 
$\th'(0) = (uv' - vu')/(u^2 + v^2)(0) = -1/a < 0$.
By~\eqref{rho theta cone}, $\th'$ does not change sign on $[0, s_0]$, 
and therefore it remains negative on $[0,s_0]$.
By~\eqref{rho theta cone}, 
$|\rho'| / \rho < C_a |\th'|$ on $[0, s_0]$, and 
\[
- C_a |\th'(s)| = C_a \th'(s) 
< \frac{\rho'(s)}{\rho(s)} 
< - C_a \th'(s) = C_a |\th'(s)|
\quad \ \forall s \in [0, s_0].
\]
Integrating over $[0, s]$ leads to~\eqref{rho away from zero}.
\end{proof}

\begin{lemma} \label{lemma:theta gira}
Assume the hypotheses of Lemma~\ref{lemma:gira gira}. 
Then 
\begin{equation} \label{theta bound below}
- \th'(s) = |\th'(s)| > c_a \quad 
\forall s \in [0, s_0], 
\qquad c_a := K_a e^{- C_a \pi / 2} > 0,
\end{equation}
where $C_a, K_a$ are defined in Lemma~\ref{lemma:pallina} and~\eqref{u.18.glob}.
\end{lemma}

\begin{proof}
It follows from the first inequality in~\eqref{rho theta cone} 
and the second inequality in~\eqref{rho away from zero}. 
\end{proof}

\begin{coro} \label{continuous bar s}
Let $a > 0$. Then there exists $\bar s = \bar s_a \in (0, \pi/(2 c_a))$ such that the Cauchy problem~\eqref{geo.4}-\eqref{u.6} admits a unique solution $(u,v)\in C^2([0,\bar s])$. 
The solution $(u,v)$ satisfies $v(\bar s) = 0$, $v(s)>0$ for all $s\in[0,\bar s)$ 
and $u(s)>0$ for all $s\in(0,\bar s]$, it has finite length $\bar s$ and it does not self--intersect. The dependence of $\bar s$ on $a$ is continuous for $a\in (0,\sqrt{2}]$.
\end{coro}

\begin{proof}
Lemmas~\ref{lemma:pallina} and~\ref{lemma:theta gira} imply that the solution keeps existing, 
without self--intersections, at least until it reaches the horizontal axis $v=0$. 
Since $\th(0) = \pi/2$ and $\th(\bar s) = 0$, by~\eqref{theta bound below} one has 
$0 < \bar s < \pi/(2 c_a)$. 

It remains to prove the continuity of the function $a \mapsto \bar s_a$. 
Let $a_0 \in (0,\sqrt{2}]$, $\e_0>0$. 
It is an immediate consequence of Lemma~\ref{lipschitz a} and of the standard continuous dependence on initial data for ODEs that for all $\e_1>0$ there exists $\d_1=\d_1(\e_1) > 0$ such that for all $a \in (a_0 - \d_1, a_0 + \d_1) \cap (0,\sqrt{2}]$ the solution $(u_a,v_a)$ of the Cauchy problem~\eqref{geo.4}-\eqref{u.6} exists at least until $\bar s_{a_0}$, with
\[
|(u_a,v_a,u'_a,v'_a)(s) - (u_{a_0},v_{a_0},u'_{a_0},v'_{a_0})(s)| < \e_1 
\quad \forall s\in[0,\bar s_{a_0}].
\]
By~\eqref{rho away from zero}, this implies that 
$|\th_a (s) - \th_{a_0} (s) | < C(a_0) \e_1$ for all $s\in[0,\bar s_{a_0}]$. 
In particular, $|\th_a (\bar s_{a_0}) - \th_{a_0} (\bar s_{a_0}) | 
= |\th_a (\bar s_{a_0})| < C(a_0) \e_1$.
By continuous dependence on initial data, we deduce a uniform bound 
$|\th'_a(s)| \geq c(a_0)$ for all $s \in [ 0, \max\{\bar s_{a_0}, \bar s_a \} ]$, 
for some $c(a_0) > 0$. 
Hence we choose $\e_1 < c(a_0) \e_0 / C(a_0)$ 
and obtain $|\bar s_a - \bar s_{a_0}| < \e_0$ 
for all $a \in (a_0 - \d_1, a_0 + \d_1) \cap (0,\sqrt{2}]$.
\end{proof}

\begin{remark}
In fact, the map $a \mapsto \bar s_a$ is continuous on $a \in (0,+\infty)$. 
In Corollary~\ref{continuous bar s} the continuity is stated only on $(0,\sqrt{2}]$ 
because this is sufficient for our goal, and because the restriction to $a\leq\sqrt{2}$ 
simplifies Lemma~\ref{lipschitz a}.
\end{remark}

\begin{lemma} \label{lemma:last}
For every $\a \in [- \pi/2, 0)$ 
there exists $a > 0$ such that the solution of~\eqref{geo.5}-\eqref{init cond} 
intersects the horizontal axis $v=0$ with tangent vector $\t = (\cos\a, \sin\a)$. 
\end{lemma}

\begin{proof}
As a consequence of Corollary~\ref{continuous bar s} and of the continuous dependence on initial data for the degenerate Cauchy problem~\eqref{geo.4}-\eqref{u.6}, 
we have, in particular, that the function 
$(0,\sqrt{2}] \to [-1,1]$, $a \mapsto u'_a(\bar s_a)$ 
is continuous.
Lemma~\ref{va a zero} implies that $u'_a(\bar s_a) \to 1$ as $a\to 0$, 
while the explicit solution $(u_{\sqrt{2}},v_{\sqrt{2}})(s) 
= (\sqrt{2} \sin (s/\sqrt{2}), \sqrt{2} \cos (s/\sqrt{2}))$ shows that 
$\bar s_{\sqrt{2}} = \pi / \sqrt{2}$ and $u'_{\sqrt{2}}(\bar s_{\sqrt{2}}) = 0$. 
\end{proof}

\bibliographystyle{amsplain}
\bibliography{networkbib}

\bigskip

\begin{flushleft}

Pietro Baldi, Emanuele Haus, Carlo Mantegazza 

\medskip

Dipartimento di Matematica e Applicazioni ``Renato Caccioppoli''

Universit\`a di Napoli Federico II 

Via Cintia, 80126 Napoli, Italy

\medskip

\texttt{pietro.baldi@unina.it} 

\texttt{emanuele.haus@unina.it}

\texttt{carlo.mantegazza@unina.it}
\end{flushleft}

\end{document}